\documentclass[10pt,conference]{ieeeconf}

\usepackage{amsthm, amsmath, amsfonts, amssymb, graphicx, enumerate, mathtools, xcolor}

\allowdisplaybreaks
\newcommand{\R}{\mathbb{R}}

\newcommand{\norm}[1]{\left|\left| #1 \right|\right|}

\newtheorem{theorem}{Theorem}
\newtheorem{lemma}[theorem]{Lemma}

\title{Regulator Design for a Congested Continuum Traffic Model with App-Routing Instability}
\author{
	Stephen Chen$^{1}$, Huan Yu$^{1}$, and Miroslav Krstic$^{1}$
	\thanks{$^{1}$: Stephen Chen, Huan Yu*, and Miroslav Krstic are with the Department of Mechanical and Aerospace Engineering, University of California, San Diego, La Jolla, CA 92093-0411 USA (e-mail: stc007@ucsd.edu, huy015@ucsd.edu*, and krstic@ucsd.edu, respectively).}%
}

\begin{document}

\maketitle

\begin{abstract}
	In this paper, we propose a control design methodology for a linearized continuum traffic model in the congested regime. The continuum traffic flow on a highway is modeled using a linearized quasilinear hyperbolic partial differential equation model known as the Aw-Rascle-Zhang (ARZ) model. The linear traffic model is augmented with a novel non-local boundary condition representing car influx due to the use of routing apps such as Google Maps and Waze. The routing apps act as real-time previews for highway traffic, introducing potentially destabilizing feedback in the app-based navigation decision process, necessitating the development of a feedback controller. We first study small-time $H^1$ solutions of the linearized model with the addition of the app-routing for sufficiently small initial data. We introduce an extended, multi-tiered boundary control design based upon the method of infinite-dimensional backstepping. Using an intermediate decoupling transformation, we account for the non-local boundary condition arising from routing app feedback. We study the existence of the extended backstepping method by characterizing the existence of the companion kernels associated with the backstepping method. Finally, we study the linear ARZ model with the app-routing extension under the designed feedback, and show that for sufficiently small $H^1$ data, the equilibrium congestion solution is exponentially stable and guarantees the existence of closed-loop solutions on the infinite time interval.
\end{abstract}
\baselineskip=.96\normalbaselineskip
\section{Introduction}

In the recent years, there has been a growing interest in the mathematical modeling and control of vehicular traffic in various contexts, especially with the advent of vehicular connectivity arising in novel technologies such as V2X (vehicle-to-everything) adaptive cruise control and mobile sensors affording greater actuation and sensing in system level traffic control and estimation.

In this paper, we are particularly interested in macroscopic models, which treat traffic phenomena as a modified fluid flow. Within this continuum context, several varying models are typically studied in analysis and control problems, and are classified into first- and second-order systems. The first-order system is developed by Lighthill-Witham-Richards (LWR) \cite{lighthill1955kinematic,richards1956shock}, involving a quasilinear first-order hyperbolic partial differential equation. While widely utilized, the LWR model approximates velocity as a static map of density, and as a result, fails to capture potential oscillatory instability (known colloquially as stop-and-go traffic). The second-order model widely used in the existing literature, in contrast, exhibits dynamics in both density and velocity, which internally generate oscillations within certain conditions. This second-order model, developed by Aw and Rascle \cite{aw2000resurrection} and later extended by Zhang \cite{zhang2002non} (together, collectively known as the ARZ model) is the model that we consider for control purposes in the paper.

Because of widespread smartphone adoption, real-time navigational routing apps such as Waze and Google Maps introduce a new significant phenomenon in traffic systems. In particular, they introduce feedback in road traffic through using real-time traffic data, however, these apps generally route in a greedy manner by locally optimizing routes. Because of this behavior, potential instability arises in road traffic. In particular, if one route has significant traffic, the incoming flux of routed cars decreases due to routing alternate routes. Conversely, if a route has little to no traffic, the routing apps will begin to prescribe a larger influx of cars. It is not difficult to conjecture that a given road will begin to operate in extreme conditions, in which it may alternate between congestion and free-flow.

Some prior literature has investigated strategies for attenuating undesirable phenomena arising in traffic by utilizing feedback controls using several different actuation methods, such as ramp metering and real-time variable speed limits. We build on top of the results found in \cite{yu2019traffic}, which achieves $L^2$ stabilization around equilibrium traffic profiles via a linearization of the ARZ model. Other particularly relevant research work is found in \cite{8550563}, which studies a modified ARZ model with nonlinear dynamic boundary condition. Of particuar interest in this work is the state space in which the authors work -- the pointwise bounded condition mandated in the original work by Aw, Rascle, and Zhang can be guaranteed on the infinite time interval, unlike stability in $L^2$. We incorporate this idea in the work in this paper. Finally, in a different context, \cite{BEKIARISLIBERIS20191} studies the stabilization of traffic phenomena in a nonlinear setting using in-domain control arising in time gap settings for adaptive cruise control. Such a paradigm employs the modern technology referred to as V2X -- vehicle-to-everything communication.

In the continuum model context, we seek to apply the method of \emph{infinite-dimensional backstepping} \cite{krsticbacksteppingpde} to design a continuum feedback controller whose purpose is to attenuate instability arising from routing apps. In the realm of backstepping control design for continuum flow models, several results exist that we build upon and extend. In \cite{dimeglio2013nplus1,hu2015nm}, full-state feedback control designs for systems of fully actuated linear hyperbolic balance laws are studied. Of particular relevance is \cite{coron2013local}, which investigates the linear backstepping control design applied to a $2 \times 2$ quasilinear hyperbolic PDE, which the ARZ model is. Notably, the linear control design (which stabilizes an equilibrium solution in $L^2$) necessitates stability in $H^2$ in the nonlinear case. Of more relevance are control for hyperbolic balance laws with non-local behavior, such as the underactuated case in \cite{chen2017transportwave}, and non-local integral terms in \cite{su2017boundary}.

\section{Modeling}
\subsection{Aw-Rascle-Zhang traffic model}
The commonly studied Aw-Rascle-Zhang (ARZ) traffic model is a 1-D continuum traffic model consisting of a quasilinear $2 \times 2$ hyperbolic partial differential equation system. One state describes the evolution of density $\rho \in L^2(0,\infty;L^\infty(\R))$, while the other describes the evolution of velocity $v \in L^2(0,\infty;L^\infty(\R))$.
\begin{align}
	\partial_t \bar{\rho} + \partial_x (\bar{\rho} \bar{v})&= 0 \label{eq:arz1}\\
	\partial_t \bar{v} + (\bar{v} - \bar{\rho} p'(\bar{\rho})) \partial_x \bar{v} &= \frac{1}{\tau}(V(\bar{\rho}) - \bar{v}) \label{eq:arz2}
\end{align}
where the traffic pressure $p:\R \rightarrow \R$ is a $C^1$ function on $\R_{+}$ defined as
\begin{align}
	p(\bar{\rho}) := v_f \left(\frac{\bar{\rho}}{\rho_m}\right)^{\gamma}
\end{align}
and the equilibrium velocity profile is chosen to be Greenshield's model 
\begin{align}
	V(\bar{\rho}) = v_f - p(\bar{\rho}) \label{eq:greenshield}
\end{align}
where $v_f \in \R_{+}$ is the freeflow velocity, $\rho_m \in \R_{+}$ is the maximum road density, $\gamma \in \R_{+}$ is a parameter capturing the aggressiveness of drivers in the road segment, and $\tau \in \R_{+}$ is a relaxation parameter related to the reaction time of drivers. In letting the relaxation $\tau \rightarrow 0$, the velocity $\bar{v}$ becomes the static relation $\bar{v} = V(\rho)$, which reduces \eqref{eq:arz1},\eqref{eq:arz2} to the widely studied Lighthill-Witham-Richards (LWR) first-order model. We also consider the initial conditions
\begin{align}
	\bar{\rho}(x,0) = \bar{\rho}_0, \quad \bar{v}(x,0) = \bar{v}_0
\end{align}
with initial $H^1$ data $(\bar{\rho}_0,\bar{v}_0)$.

The well-posedness of \eqref{eq:arz1},\eqref{eq:arz2} is studied in \cite{aw2000resurrection,zhang2002non}, in which the ARZ model is shown to be strictly hyperbolic \emph{except} for cases near vacuum ($\bar{\rho} = 0$) where the eigenvalues of the convection operator coalesce. We consider solutions in the admissible region $\mathcal{R} = \{ (\bar{\rho}(t),\bar{v}(t)) \in L^2(D) | 0 < \bar{\rho}(t) \leq \rho_m, 0 < v(t) \leq \bar{v} - p(\bar{\rho}) \}$.

\subsection{Linear ARZ}
As with \cite{yu2019traffic}, we will first apply a change of variables from $(\bar{v},\bar{\rho})$ to $(\bar{v},\bar{q})$, with $\bar{q} = \bar{\rho} \bar{v}$ (also known as the \emph{flow rate}). The flow rate $\bar{q}$ and the velocity $\bar{v}$ will follow the system dynamics governed by
\begin{align}
	\partial_t \bar{q} + v \partial_x \bar{q} &= \frac{\bar{q}(\gamma p - \bar{v})}{\bar{v}} \partial_x \bar{v} + \frac{v_f - p - \bar{v}}{\tau \bar{v}}\bar{q} \label{eq:arz_qv_q}\\
	\partial_t \bar{v} - (\gamma p - \bar{v}) \partial_x \bar{v} &= \frac{v_f - p - \bar{v}}{\tau} \label{eq:arz_qv_v}
\end{align}
which follows directly from \eqref{eq:arz1},\eqref{eq:arz2}. The traffic pressure $p$ can be rewritten in the $(\bar{v},\bar{q})$-coordinates as
\begin{align}
		p = \frac{v_f}{\rho_m^{\gamma}} \left( \frac{\bar{q}}{\bar{v}} \right)^{\gamma}
\end{align}
and the intial condition for $q$ is given by
\begin{align}
	\bar{q}(x,0) &= \rho_0 v_0
\end{align}
We will study the system around constant equilibrium flow rate $q^*$ and velocity $v^*$. These chosen equilibrium states are not independent -- they must satisfy the following compatibility condition:
\begin{align}
	q^* = \rho_m v^* \sqrt[\gamma]{\frac{v_f - v^*}{v_f}} \label{eq:equib_qv}
\end{align}
which arises due to \eqref{eq:greenshield}, and is consistent with a stationary point in the transformed model \eqref{eq:arz_qv_v}. By investigating sufficiently small pertubations about $(v^*,q^*)$
\begin{align}
	\tilde{q} := \bar{q} - q^*, \quad \tilde{v} := \bar{v} - v^*
\end{align}
we can find linearized dynamics for flow rate and velocity as
\begin{align}
	\partial_t \tilde{q} + v^* \partial_x \tilde{q} &= \frac{q^*(\gamma p^* - v^*)}{v^*} \partial_x \tilde{v} \nonumber\\&\quad+ \frac{q^*}{\tau} \left(\frac{1}{v^*} - \frac{1}{\gamma p^*} \right) \tilde{v} - \frac{\gamma p^*}{\tau v^*} \tilde{q} \label{eq:arz_qvlin_q}\\
	\partial_t \tilde{v} - (\gamma p^* - v^*) \partial_x \tilde{v} &= \frac{\gamma p^* - v^*}{\tau v^*} \tilde{v} - \frac{\gamma p^*}{\tau v^*} \tilde{q} \label{eq:arz_qvlin_v}
\end{align}
In particular, we will study solutions around the so-called \emph{congestion} regime, that is, an equilibrium profile $(v^*,q^*)$ satisfying the following condition:
\begin{align}
    0 < v^* < \frac{\gamma}{\gamma + 1} v_f \label{eq:cong_cond}
\end{align}

We utilize the following transformation to diagonalize the convection operator:
\begin{align}
	w &:= \tilde{q} - q^*\left(\frac{1}{v^*} - \frac{1}{\gamma p^*} \right) \tilde{v} \label{eq:w_def}\\
	v &:= \frac{q^*}{\gamma p^*} \tilde{v} \label{eq:v_def}
\end{align}
The transformation admits the $(v,w)$-system:
\begin{align}
	\partial_t w &= -v^* \partial_x w \label{eq:model_w}\\
	\partial_t v &= (\gamma p^* - v^*) \partial_x v + c(x) w \label{eq:model_v}
\end{align}
subject to boundary conditions
\begin{align}
	w(0,t) &= k_1 v(0,t) + U_{\textrm{rout}}(t) \label{eq:model_w_bc}\\
	v(L,t) &= k_2 w(L,t) + U_{\textrm{ramp}}(t) \label{eq:model_v_bc}
\end{align}
and initial conditions
\begin{align}
    w(x,0) = w_0, \quad v(x,0) = v_0
\end{align}
The coefficients are given by
\begin{align}
	c(x) &= -\frac{1}{\tau} \exp\left(-\frac{x}{\tau v^*}\right)\\
	k_1 &= \frac{\gamma p^* - v^*}{v^*}\\
	k_2 &= \exp\left(-\frac{L}{\tau v^*}\right)
\end{align}
Note that one can determine the congestion condition \eqref{eq:cong_cond} to enforce the following conditions on convection speeds:
\begin{align}
    v^* > 0, \quad \gamma p^* > v^*
\end{align}

The system \eqref{eq:model_w},\eqref{eq:model_v} represents a highway segment on the domain $(0,L)$, where $w$ represents a deviation around an equilibrium flow profile $w^*$, while $v$ represents a deviation around an equilibrium velocity profile $v^*$. There are two controls present at $x = 0, x = L$.
\begin{enumerate}
	\item $U_{\textrm{rout}}(t)$: Represents the influence that the routing app has on the influx of cars. One may conjecture that if the road is more congested, $U_{\textrm{rout}}(t)$ should decrease, and conversely, the opposite true if the road exhibits more free-flow characteristics.
	\item $U_{\textrm{ramp}}(t)$: Represents on-ramp metering control of influx, to be designed.
\end{enumerate}

We postulate a strict-feedback representation of the app-routing feedback--in general, this routing feedback may be extremely complex and include feedforward as well as feedback. However, the system \eqref{eq:model_w},\eqref{eq:model_v_bc} will lose its strict-feedback structure. Thus, as an initial result, we assume (heuristically) the following destabilizing app-routing feedback that preserves strict-feedback:
\begin{align}
	U_{\textrm{rout}}(t) &= \int_0^L a(y) w(y,t) dy
\end{align}
One can intepret the routing app to be an \emph{adversarial} feedback controller that potentially destabilizes the system.

\begin{figure}[tb]
		\centering
		\includegraphics[width=0.8\linewidth]{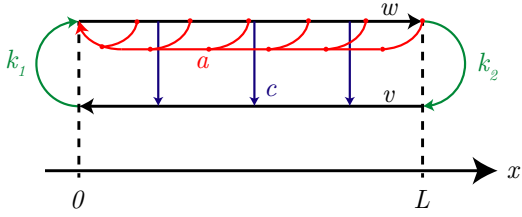}
		\caption{Schematic representation of the Riemann invariants of the linear ARZ model. Potentially destabilizing coupling appears in two ways: via the interior coupling ${\color{blue} c}$, and/or via the routing feedback ${\color{red} a}$.}
		\label{fig:system_schem}
\end{figure}

\subsection{Local well-posedness of linear ARZ with routing feedback}
Before feedback control design, we must first determine the notion of an admissible solution to the linear ARZ model. We note, in particular, that several boundedness conditions must be satisfied pointwise:
\begin{align}
	0 < \bar{q} < q_{\textrm{max}} = \rho_m \bar{v}, \quad 0 < \bar{v} < v_f, \quad \forall x \in [0,L] \label{eq:qv_bd}
\end{align}
These box constraints must be met for every spatial point. The first obvious condition is that $\bar{q},\bar{v}$ must be positive quantities. In fact, the quasilinear ARZ \eqref{eq:arz1},\eqref{eq:arz2} exhibits vanishing coefficients at vacuum ($\bar{\rho} = 0$), which is problematic for establishing existence of solutions (one can note that the characteristics coalesce when $\bar{\rho} \rightarrow 0$). The second condition $\bar{q} < q_{\textrm{max}}$, $\bar{v} < v_f$ imposes pointwise maximum values on flow rate and velocity. In particular, we do not desire the pointwise flow rate to exceed a value associated with maximum road density. If density exceeds this maximum at a point, then one may expect traffic collisions, which are undesireable.

From these observations, it becomes clear that it is desireable to seek pertubation solutions $(v,q)$ that are sufficiently bounded in the sense of the $L^\infty$ norm. That is, we impose the following conditions:
\begin{align}
	\norm{\tilde{q}}_{L^\infty} &< \min\{ q_{\textrm{max}} - q^*, q^* \} \\
	\norm{\tilde{v}}_{L^\infty} &< \min\{ v_f - v^*, v^* \}
\end{align}
which we conservatively combine into one condition
\begin{align}
	\norm{(\tilde{v},\tilde{q})}_{L^\infty} < \varepsilon_0 := \min \{ q_{\textrm{max}} - q^*, q^*, v_f - v^*, v^* \} \label{eq:qv_bnd}
\end{align}

As opposed with control design in previous literature, we recognize that the space of $L^2$ functions is insufficient to capture pointwise boundedness. Thus, we derive energy growth estimates in $H^1$ for the linearized ARZ model with $U_\textrm{ramp} = 0$. The model necessitates a set $\bar{S}$ of admissible states defined by $\emph{pointwise}$ boundeded ($L^\infty$) box constraints:
\begin{align}
	\bar{S} = \{ (v,w) \in L^{\infty} | \norm{(v,w)}_{L^\infty} < \varepsilon, \forall x \in [0,L] \}
\end{align}
where $\varepsilon$ can be computed from $\varepsilon_0$ from \eqref{eq:w_def},\eqref{eq:v_def},\eqref{eq:qv_bnd}. We restrict the set $\bar{S}$ to that of an $H^1$-bounded subset $S$. Note that this restriction is conservative, since the bound on the pointwise energy $\norm{(v,w)}_{L^\infty}$ becomes less sharp. However, the analysis in $H^1$ is simpler than that of $L^\infty$, therefore, as an initial result, we utilize this restriction. We define the subset $S$:
\begin{align}
	S = \{ (v,w) \in H^1 | \norm{(v,w)}_{H^1} < \sqrt{L} \varepsilon, \forall x \in [0,L] \}
\end{align}
One can note that $S \subset \bar{S}$ from Sobolev embedding, where the factor of $\sqrt{L}$ arises from the Sobolev embedding inequality in one dimension (note: there may be more optimal factors to incorporate more admissible initial conditions). Thus, seeking solutions in the set $S$ is sufficient to guarantee existence of solutions in $\bar{S}$, but is particularly restrictive.

It is worth to note that in \cite{8550563}, the authors study control for a different, nonlinear model with the same considerations on pointwise bounded states, and seek solutions in the Sobolev space $W^{2,\infty}$, which is naturally endowed with the $L^\infty$ norm. In our case, however, in the context of a purely linear ARZ model, we seek pointwise bounded states by considering the state space in $H^1$, which is more conservative but tractable for the Lyapunov arguments we utilize. However, if one is to consider the linear control design applied to the nonlinear ARZ model akin to the work in \cite{coron2013local}, one may need to seek $W^{2,\infty}$ solutions (or conservatively, the embedded space $H^3$) to achieve pointwise boundedness.

\begin{lemma}[$H^1$ energy estimates]
	\label{thm:h1_energy}
	Consider the strictly hyperbolic system \eqref{eq:model_w},\eqref{eq:model_v} and associated initial and boundary conditions. Then the following $H^1$ energy estimates hold:
	\begin{align}
		\norm{(v,w)}_{H^1}(t) \leq C(t) \norm{(v_0,w_0)}_{H^1}
	\end{align}
	where $C:[0,T) \rightarrow \R_+$. Then for sufficiently small initial data $(v_0,w_0)$ there exists a $T > 0$ such that
	\begin{align}
		\norm{(v,w)(\cdot,t)}_{H^1} < \varepsilon \label{eq:state_bnd}
	\end{align}
	for all $t \in [0,T)$. The finite time $T$ will depend on the parameters $a,c,k_1,k_2$, and initial data $(v_0,w_0)$, as well as the bound $\varepsilon$.
\end{lemma}

We omit the full proof for space in this paper, but will sketch the methodology. We first apply the transformation
\begin{align}
		\omega(x,t) &= w(x,t) - \int_x^L a(y-x) w(y,t) dy \label{eq:estimates_tfm}
\end{align}
and seek $H_1$ growth bounds for $(v,\omega)$. One takes $L^2$ inner products of the dynamic equations for $(v,\omega)$ with the states $(v,\omega)$ and integrate by parts to generate $L^2$ estimates. The dynamic equations for $(v,\omega)$ are differentated once in $x$, and the previous method is repeated for $(\partial_x v, \partial_x \omega)$, which generate $L^2$ estimates in the spatial derivatives. The two $L^2$ estimates are combined to derive an $H^1$ estimate for $(v,\omega)$. $H^1$ equivalence in norm is established between $(v,w)$ and $(v,\omega)$ using \eqref{eq:estimates_tfm} and the regularity properties of $a$, thus generating an $H^1$ energy estimate for $(v,w)$.

\section{Feedback design for linearized system}
	%
	We first utilize the following backstepping transformation to shift the interior term $c(x) w(x,t)$ in \eqref{eq:model_w}-\eqref{eq:model_v_bc} into the $x = L$ boundary, where it can be neutralized with the boundary controller.\
	\begin{align}
		z(x,t) &= v(x,t) - \int_0^x k(x,y) v(y,t) dy \nonumber\\
		&\qquad\qquad- \int_0^x l(x,y) w(y,t) dy \nonumber\\
		&\qquad\qquad- \int_x^L m(x,y) w(y,t) dy \label{eq:tfm1}
	\end{align}
	The kernels of transformation $k,l \in C^1(\mathcal{T}_l), m \in C^(\mathcal{T}_u)$ are to be determined, where $\mathcal{T}_l = \{ (x,y) \in \R^2 | 0 \leq y \leq x \leq L \}$ and $\mathcal{T}_u = \{ (x,y) \in R^2 | 0 \leq x \leq y \leq L \}$.

	This transformation is an extension to the standard infinite-dimensional backstepping technique, which typically does not exhibit the so called \emph{forwarding} transformation, which appears as the Volterra integral operator characterized by the kernel $m$. However, despite this non-standard structure, the transformation \eqref{eq:tfm1} remains a Volterra integral transformation of the second kind, as the forwarding integral enters in an affine manner in $v \leftrightarrow z$.

	\eqref{eq:tfm1} admits the following intermediate target system.
	\begin{align}
		\partial_t w(x,t) &= -v^* \partial_x w(x,t) \label{eq:model_tfm_w}\\
		\partial_t z(x,t) &= (\gamma p^* - v^*) \partial_x z(x,t) \label{eq:model_tfm_z} \\
		w(0,t) &= k_1 z(0,t) + \int_0^L a(x) w(x,t) dx \label{eq:model_tfm_w_bc}\\
		z(L,t) &= V_{\textrm{ramp}}(t) \label{eq:model_tfm_z_bc}
	\end{align}
	where $V_{\textrm{ramp}(t)}$ is defined to be
	\begin{align}
		V_{\textrm{ramp}}(t) &:= U_{\textrm{ramp}}(t) + k_2 w(L,t) - \int_0^L k(L,y) v(y,t) dy \nonumber\\
		&\qquad\qquad\qquad\qquad\qquad- \int_0^L l(L,y) w(y,t) dy \label{eq:ps_control1}
	\end{align}
	The kernels of transformation $k,l,m$ must satisfy a set of coupled conditions that comprise the following hyperbolic PDE system:
	\begin{align}
		\partial_x k(x,y) + \partial_y k(x,y) &= 0 \label{eq:k_pde} \\
		(\gamma p^* - v^*) \partial_x l(x,y) - v^* \partial_y l(x,y) &= c(y) k(x,y) \nonumber\\&\quad+ v^* l(x,0) a(y) \label{eq:l_pde} \\
		(\gamma p^* - v^*) \partial_x m(x,y) - v^* \partial_y \partial_y m(x,y) &= v^* l(x,0) a(y) \label{eq:m_pde}
	\end{align}
	subject to the boundary conditions
	\begin{align}
		k(x,0) &= \frac{k_1 v^*}{\gamma p^* - v^*} l(x,0) \label{eq:k_bc} \\
		l(x,x) &= m(x,x) - \frac{c(x)}{\gamma p^*} \\
		m(x,L) &= 0 \\
		m(0,y) &= 0 \label{eq:m_bc2}
	\end{align}
	The wellposedness of this PDE system will be studied in a later section.


	The target system \eqref{eq:model_tfm_w}-\eqref{eq:model_tfm_z_bc} still exhibits feedback. However, due to the first transformation \eqref{eq:tfm1}, the recirculatory feedback now appears \emph{strictly} in the boundary. We will apply a series of two moree invertible transformations to attenuate this recirculation.

	We will define the following parameter $\mu^*$:
	\begin{align}
		\mu^* = \frac{v^*}{\gamma p^* - v^*}
	\end{align}
	The parameter $\mu^*$ is the ratio of transport speeds between $v,w$.

	The app routing feedback, a nonstandard recirculation behavior, motivates the use of a new, piecewise transformation that shifts the nonlocal boundary coupling to that of a trace coupling existing on the interior.

	We define the following piecewise transformation from $z \leftrightarrow \eta$
	\begin{align}
		\eta(x,t) =
		\begin{dcases}
			\begin{array}{l} k_1 z(x,t) \\ \,\,+ \int_{\mu^* x}^{L} a(y) w(y-\mu^* x,t) dy \end{array}  & x \in \left[0,\frac{L}{\mu^*}\right] \\
			k_1 z(x,t) & \text{otherwise}
		\end{dcases}
		\label{eq:decoup_tfm}
	\end{align}

	Computing the time derivative of this transformation, one finds
	\begin{align}
		\partial_t \eta(x,t) =
		\begin{dcases}
		 	\begin{array}{l} k_1 (\gamma p^* - v^*) \partial_x z(x,t) \\ \quad- \Omega[w](x) \end{array} & x \in \left[0,\frac{L}{\mu^*}\right] \\
			k_1 (\gamma p^* - v^*) z_x(x,t) & \text{otherwise}
		\end{dcases}
		\label{eq:decoup_tfm_t}
	\end{align}
	where the operator $\Omega$ is defined
	\begin{align}
		\Omega[w](x) := \int_{\mu^* x}^{L} v^* a(y) \partial_y w(y-\mu^* x,t) dy
	\end{align}

	Computing the spatial derivative on $(0,L/\mu^*) \cup (L/\mu^*,L)$,
	\begin{align}
		\partial_x \eta(x,t) =
		\begin{dcases}
			\begin{array}{l} k_1 \partial_x z(x,t) \\ \quad - (\gamma p^* - v^*)^{-1} \Omega[w](x) \\ \quad - a(\mu^* x) w(0,t) \end{array} & x \in \left[0,\frac{L}{\mu^*}\right] \\
			k_1 \partial_x z(x,t) & \text{otherwise}
		\end{dcases}
		\label{eq:decoup_tfm_x}
	\end{align}
	We compute this spatial derivative almost everywhere, where a potential jump discontinuity in $z_x$ may appear at $x = L/\mu^*$. This, however, is not an issue -- the solution space we consider is $H^1$, where piecewise differentiability is admissible.

	By combining \eqref{eq:decoup_tfm_t},\eqref{eq:decoup_tfm_x}, the following evolution equation for $\eta(x,t)$ is found:
	\begin{align}
		\eta_t(x,t) &= (\gamma p^* - v^*) \eta_x(x,t) + \check{a}(x) w(0,t)
	\end{align}
	where the parameter $\check{a}(x)$ is defined piecewise as:
	\begin{align}
		\check{a}(x) :=
		\begin{dcases}
			(\gamma p^* - v^*) a(\mu^* x) & x \in [0,L/\mu^*] \\
			0 & \text{otherwise}
		\end{dcases}
		\label{eq:a_chk_def}
	\end{align}

	The $(w,\eta)$ system can then be expressed as
	\begin{align}
		\partial_t w(x,t) &= -v^* \partial_x w(x,t) \label{eq:model_tfm2_w}\\
		\eta_t(x,t) &= (\gamma p^* - v^*) \eta_x(x,t) + \check{a}(x) \eta(0,t) \label{eq:model_tfm2_eta} \\
		w(0,t) &= \eta(0,t) \label{eq:model_tfm2_w_bc}\\
		\eta(L,t) &= W_{\textrm{ramp}}(t) \label{eq:model_tfm2_eta_bc}
	\end{align}
	where $W_{\text{ramp}}(t)$ is found from evaluating the transform \eqref{eq:decoup_tfm} at $x = L$, noting that depending on the case $\mu^* \leq 1,\mu^* > 1$ one has a piecewise representation which we have compactly formulated as:
	\begin{align}
		W_{\text{ramp}}(t) &:= k_1 V_{\text{ramp}}(t) + \int_{\min\{\mu^* L,L\}}^{L} a(y) w(y-\mu^* x,t) dy \label{eq:ps_control2}
	\end{align}

	The final step involves a single backstepping transformation from $\eta(x,t)$ to $\xi(x,t)$. The target dynamic for $\xi$ is formulated as
	\begin{align}
		\xi_t(x,t) &= \xi_x(x,t) \\
		\xi(L,t) &= 0
	\end{align}
	which is achieved through the backstepping transformation
	\begin{align}
		\xi(x,t) &= \eta(x,t) - \int_0^x n(x-y) \eta(y,t) dy \label{eq:tfm3}
	\end{align}
	This leads us to our final target system $(w,\xi)$, which is trivially finite-time stable.
	\begin{align}
		\partial_t w(x,t) &= -v^* \partial_x w(x,t) \label{eq:model_tfm3_w}\\
		\xi_t(x,t) &= (\gamma p^* - v^*) \xi_x(x,t) \label{eq:model_tfm3_eta} \\
		w(0,t) &= \xi(0,t) \label{eq:model_tfm3_w_bc}\\
		\xi(L,t) &= 0 \label{eq:model_tfm3_eta_bc}
	\end{align}
	The kernel of transformation \eqref{eq:tfm3} must satisfy the following Volterra integral equation of the second kind:
	\begin{align}
		n(x) &= (v^* - \gamma p^*)^{-1} \left[ \check{a}(x) + \int_0^x \check{a}(y)n(x-y) dy \right] \label{eq:n_ie}
	\end{align}
	Finally, the controller $W_\textrm{ramp}(t)$ can be found by evaluating transform \eqref{eq:tfm3} at $x = L$:
	\begin{align}
		W_\textrm{ramp}(t) &= \int_0^L n(L-y) \eta(y,t) dy \label{eq:ps_control3}
	\end{align}
	By combining \eqref{eq:ps_control1},\eqref{eq:ps_control2},\eqref{eq:ps_control3}, the feedback controller $U_\textrm{ramp}(t)$ can be found:
	\begin{align}
		U_\textrm{ramp}(t) &= a(t) -k_2 w(L,t) + \int_0^L F_v(y) v(y,t) dy \nonumber\\
		&\qquad\qquad\qquad+ \int_0^L F_w(y) w(y,t) dy \nonumber\\
		&\quad - \int_{\min\{\mu^* L,L\}}^{L} k_1^{-1} a(y) w(y-\mu^* x,t) dy \label{eq:fdback}
	\end{align}
	where $F_v(y), F_w(y)$ are given by the following relations:
	\begin{align}
		F_v(y) &= k(L,y) + n(L-y) - \int_y^1 n(L-\xi) k(\xi,y) d\xi \\
		F_w(y) &= l(L,y) + \int_0^{\frac{L-y}{\mu^*}} k_1^{-1} n(L-\xi) a(y+\mu^* \xi) d\xi \nonumber\\&\qquad\qquad- \int_y^1 n(L-\xi) l(\xi,y) d\xi \nonumber\\&\qquad\qquad- \int_0^y n(L-\xi) m(\xi,y) d\xi
	\end{align}
	and $a(t)$ is a dynamic extension with dynamics defined as
	\begin{align}
		\dot{a}(t) &= -k_3 a(t)
	\end{align}
	where $k_3 > 0$ is chosen, and the initial condition $a(0)$ is chosen to fulfill compatibility conditions between the initial control value $U(0)$ and the initial condition $v_0$. It is important to note that the dynamically extended state is trivially exponentially stable, and therefore does not compromise the (exponential) stability of the system.

	We state the main theorem of the paper below:
	\begin{theorem}
		The boundary controller \eqref{eq:fdback} $H^1$-exponentially stabilizies the zero solution of the linear pertubation ARZ model \eqref{eq:arz_qvlin_q},\eqref{eq:arz_qvlin_v}. That is, there exists $M,\gamma \in \R_+$ such that
		\begin{align}
			\norm{(v,w)}_{H^1} \leq M \exp (-\gamma t) \norm{(v_0,w_0)}_{H^1} \label{eq:stability_ineq}
		\end{align}
		Moreover, \eqref{eq:fdback} ensures the global (in time) existence of solutions for sufficiently small initial conditions $\norm{(v_0,w_0)}_{H^1} < M^{-1} \varepsilon$, where \eqref{eq:stability_ineq} generates a priori $H^1$ energy estimates for the solution $(v,w)$.
		\label{thm:main}
	\end{theorem}

	The full proof is not given in the interest of space, but the argument is a small extension to the standard method in previous backstepping literature. One utilizes an $H^1$ stability estimate of the target system \eqref{eq:model_tfm3_w}-\eqref{eq:model_tfm3_eta_bc} coupled with invertible (and regularity-preserving) backstepping transformations \eqref{eq:tfm1},\eqref{eq:decoup_tfm},\eqref{eq:tfm3} to derive $H^1$ stability for $(v,w)$. These properties are studied briefly in Section \ref{sec:stability}.

	To enforce the global existence of solutions, we employ an argument similar to Theorem \ref{thm:h1_energy}, although instead of seeking $H^1$ energy estimates, we employ our $H^1$ stability estimate \eqref{eq:stability_ineq} instead, which, in fact, are decaying energy estimates. We seek small $H^1$ initial conditions such that \eqref{eq:state_bnd} is enforced on the time interval $(0,\infty)$. A sufficient set of admissible initial conditions can be found as the $H^1$ neighborhood $\{(v_0,w_0) \in H^1 | \norm{(v_0,w_0)}_{H^1} < M^{-1} \varepsilon\}$.

	\section{Closed-loop stability in $H^1$}
	\label{sec:stability}
		To show the feedback controller \eqref{eq:fdback} exponentially stabilizes the pertubation equilibrium $(v^*,q^*)$, we utilize a Lyapunov argument. It is not only necessary to prove convergence, but also to prove pointwise boundedness in our stability estimate. To ensure existence of solutions to the closed-loop model, the states cannot exceed a threshold either arising due to the linear model approximation failing beyond the domain of attraction, and/or the state exceeding physical flow rate/velocity constraints (positivity, maximum capacity, speed limits).

		We will give a series of lemmas that establish $H^1$ stability in the target system and establish equivalence in $H^1$ norm between all transformed states. The combination of the following lemmas will admit the result Theorem \ref{thm:main}. We will forgo some of the more straightforward proofs, and instead sketch the argument.

		\begin{lemma}
			The zero solution of the target system \eqref{eq:model_tfm3_w}-\eqref{eq:model_tfm3_eta_bc} is exponentially stable in the sense of $H^1$, that is,
			\begin{align}
				\norm{(\eta,w)}_{H^1} \leq M_1 \exp(-\gamma t) \norm{(\eta_0,w_0)}_{H^1}
			\end{align}\label{lem:targ_es}
		\end{lemma}
		Lemma \ref{lem:targ_es} is not difficult to see. One can utilize the following Lyapunov function
		\begin{align}
			V(t) &= \int_0^L \bigg[ e^{-\delta_1 x} w(x,t)^2 + d_1 e^{\delta_2 x} v(x,t)^2 \nonumber\\&\qquad+ e^{-\delta_3 x} \partial_x w(x,t)^2 + d_2 e^{\delta_4 x} \partial_x v(x,t)^2 \bigg] dx
		\end{align}
		where the coefficients are chosen $\delta_i > 0, i \in \{1,...,4\}$, and $d_1 > \mu^*, d_2 > 1/\mu^*$. The coefficients $M_1, \gamma$ are then given by
		\begin{align}
			M_1 &= \frac{\max\{e^{\delta_2 L},e^{\delta_4 L}\}}{\min\{e^{-\delta_1 L},e^{-\delta_3 L}\}} \\
			\gamma &= \frac{1}{4} \min \{ \delta_1 v^*, \delta_3 v^*, d_1 \delta_2 (\gamma p^* - v^*), d_2 \delta_4(\gamma p^* - v^*) \}
		\end{align}

		\begin{lemma}
			The transformation \eqref{eq:tfm3} and its associated inverse transform establishes $H^1$ equivalence in norm between $(\xi,w)$ and $(\eta,w)$, i.e. there exist $C_1, C_2 > 0$ such that
			\begin{align}
				C_1 \norm{(\xi,w)}_{H^1} \leq \norm{(\eta,w)}_{H^1} \leq C_2 \norm{(\xi,w)}_{H^1}
			\end{align}\label{lem:tfm3_h1}
		\end{lemma}
		\begin{proof}[Partial proof]
		We will derive the estimate $C_1 \norm{(\xi,w)}_{H^1} \leq \norm{(\eta,w)}_{H^1}$. The computation of the upper estimate is identical, but employs the inverse transformation.
		
		The equivalence in $H^1$ norm is a nontrivial extension from equivalence in $L^2$ norm typically found in backstepping literature, since the kernel of transformation \eqref{eq:tfm3} has, in general, $L^2$ regularity, and may not be continuous. 

		We compute the $H^1$ bound on $\xi$ using \eqref{eq:tfm3} and the triangle inequality property of norms:
		\begin{align}
			\norm{\xi}_{H^1} \leq \norm{\eta}_{H^1} + \norm{\int_0^x n(x-y) \eta(y,t) dy}_{H^1} \label{eq:est_1}
		\end{align}
		From the definition of the $H^1$ norm,
		\begin{align}
			&\norm{\int_0^x n(x-y) \eta(y,t) dy}_{H^1} \nonumber\\
			&\quad= \bigg( \norm{\int_0^x n(x-y) \eta(y,t) dy}_{L^2}^2 \nonumber\\
			&\qquad\quad+ \norm{\partial_x \left[\int_0^x n(x-y) \eta(y,t) dy \right]}_{L^2}^2 \bigg)^{\frac{1}{2}} \label{eq:est_2}
		\end{align}
		The first term is easily bounded using Cauchy-Schwarz, and the fact that $x \in [0,L]$. The second term, however, requires a more involved analysis.
		\begin{align}
			&\norm{\partial_x \left[\int_0^x n(x-y) \eta(y,t) dy \right]}_{L^2} \nonumber\\
			&\quad= \norm{ n(0) \eta - \int_0^x n'(x-y) \eta(y,t) dy }_{L^2}
		\end{align}
		Note that because of the discontinuity in $n$, the derivative $n'$ is unbounded. However, since $n$ is a special convolution kernel, one can utilize integration by parts to shift the derivative, apply the triangle and Cauchy-Schawrz inequalities, and establish an estimate. The integration by parts yields
		\begin{align}
			&\norm{\partial_x \left[\int_0^x n(x-y) \eta(y,t) dy \right]}_{L^2} \nonumber\\
			&\quad= \norm{ n(x) \eta(0,t) + \int_0^x n(x-y) \partial_y \eta(y,t) dy }_{L^2} \nonumber\\
			&\quad\leq |\eta(0,t)| \norm{n}_{L^2} + \norm{n}_{L^2} \norm{\partial_x \eta}_{L^2}
		\end{align}
		By using a Sobolev embedding inequality (Agmon or Morrey's inequality in $\R$), one can bound the trace term $\eta(0,t)$ via an $H^1$ estimate on $\partial_x \eta$:
		\begin{align}
			&\norm{\partial_x \left[\int_0^x n(x-y) \eta(y,t) dy \right]}_{L^2} \nonumber\\
			&\quad\leq 2 \norm{n}_{L^2} \norm{\eta}_{H^1} \label{eq:est_3}
		\end{align}
		Combining \eqref{eq:est_1},\eqref{eq:est_2},\eqref{eq:est_3}, one finds the following estimate:
		\begin{align}
			\norm{\xi}_{H^1} &\leq \norm{\eta}_{H^1} + \norm{n}_{L^2} \norm{\eta}_{L^2} + 2 \norm{n}_{L^2} \norm{\eta}_{H^1} \nonumber \\
			&\leq (1 + 3 \norm{n}_{L^2})\norm{\eta}_{H^1}
		\end{align}
		Noting that $w$ is purely an identity transform, it is easy to see that one finds the bound
		\begin{align}
			C_1 \norm{(\xi,w)}_{H^1} \leq \norm{(\eta,w)}_{H^1}
		\end{align}
		with
		\begin{align}
			C_1 = (1 + 3 \norm{n}_{L^2})^{-1}
		\end{align}
		We stress that due to the convolution structure of the kernel $n$, it is sufficient to seek $L^2$ bounds on $n$ (studied in Section \ref{sec:kernel}) for $H^1$ norm equivalence.

	\end{proof}

		\begin{lemma}
			The transformation \eqref{eq:decoup_tfm} and its associated inverse transform establishes $H^1$ equivalence in norm between $(\eta,w)$ and $(z,w)$, i.e., there exist $C_3,C_4 > 0$ such that
			\begin{align}
				C_3 \norm{(\eta,w)}_{H^1} \leq \norm{(z,w)}_{H^1} \leq C_4 \norm{(\eta,w)}_{H^1}
			\end{align}\label{lem:tfm2_h1}
		\end{lemma}
		For \eqref{eq:decoup_tfm}, the equivalence is quite simple to see. \eqref{eq:decoup_tfm} is purely a piecewise affine transformation.

		\begin{lemma}
			The transformation \eqref{eq:tfm1} and its associated inverse transform establishes $H^1$ equivalence in norm between $(z,w)$ and $(v,w)$, i.e., there exist $C_5,C_6 > 0$ such that
			\begin{align}
				C_5 \norm{(z,w)}_{H^1} \leq \norm{(v,w)}_{H^1} \leq C_6 \norm{(z,w)}_{H^1}
			\end{align}\label{lem:tfm1_h1}
		\end{lemma}
		\eqref{eq:tfm1} is a classical backstepping transformation, and thus equivalent norm properties are studied in \cite{krsticbacksteppingpde}.

	\section{Existence of solutions to kernel equations}
	\label{sec:kernel}
	For the backstepping transformations \eqref{eq:tfm1} and \eqref{eq:tfm3} to exist, both the companion boundary value problem given by \eqref{eq:k_pde}-\eqref{eq:m_pde} with boundary conditions \eqref{eq:k_bc}-\eqref{eq:m_bc2} and the integral equation \eqref{eq:n_ie} must have solutions.

	We will begin by studying the existence of solutions to \eqref{eq:k_pde}-\eqref{eq:m_pde}.

	\begin{lemma}
		Consider the boundary value problem given by the system of hyperbolic equations \eqref{eq:k_pde}-\eqref{eq:m_pde} with boundary conditions \eqref{eq:k_bc}-\eqref{eq:m_bc2}. Assume that $a,c \in C([0,L])$. Then there exist unique solutions $k,l \in C(\mathcal{T}_l)$ and $m \in C(\mathcal{T}_u)$.
	\end{lemma}

	\begin{proof}
		The proof follows from finding a solution via the method of characteristcs. Since there is interesting non-local behavior arising in the plant, the companion kernel PDEs will exhibit non-local coupling as well, which must be treated carefully in the characteristics.
		
		\begin{figure}[tb]
		  \centering
			\includegraphics[width=\linewidth]{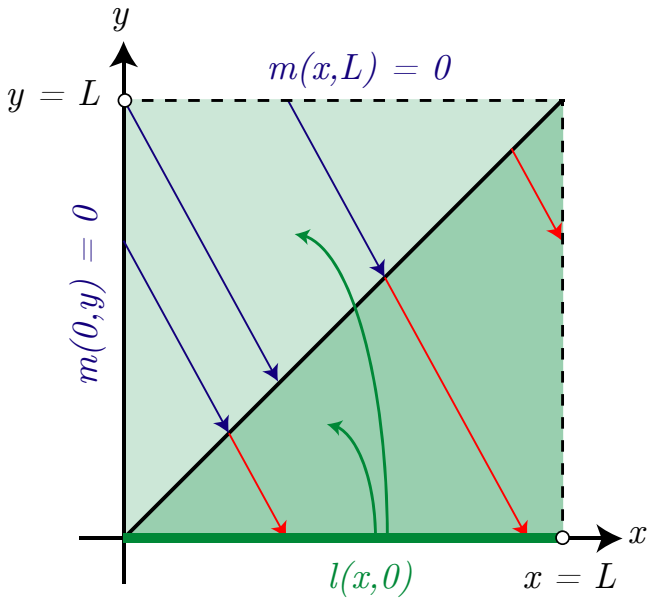}
			\caption{Characteristic lines for $(l,m)$ companion kernel PDE. A transmission condition between the two kernels appears at $y = x$. The boundary value $l(x,0)$ introduces feedback via trace terms in the evolution equation for both $l,m$.}
			\label{fig:characteristics}
	\end{figure}

		We begin by inspecting the $k$-PDE. By directly applying the characteristics method, it is not hard to see that $k$ has the following representation for a solution in $\mathcal{T}_l$:
		\begin{align}
			k(x,y) &= k(x-y,0) = k_1 \mu^* l(x-y,0) \label{eq:k_soln}
		\end{align}
		By using \eqref{eq:k_soln} as a representation for $k$ in \eqref{eq:l_pde}, one can find the self-contained system $(l,m)$. The $(l,m)$ characteristics are sketched in Figure \ref{fig:characteristics}. In particular, note how boundary conditions between $l,m$ are coupled, in particular, at $y = x$ and $y = 0$.
		
		By a direct application of the method of characteristics to \eqref{eq:l_pde}, one can recover an integral equation representation for $l$:
		\begin{align}
			l(x,y) &= m(\sigma_1(x,y),\sigma_1(x,y)) - \frac{c(\sigma_1(x,y))}{\gamma p^*} \nonumber\\&\quad+ \int_0^{\frac{x-y}{\gamma p^*}} \bigg[k_1 \mu^* c(-v^* s + \sigma_1(x,y)) l(\gamma p^* s, 0) \nonumber\\&\qquad\qquad\qquad+ v^* a(-v^*s + \sigma_1(x,y)) \nonumber\\&\qquad\qquad\qquad\quad \times l((\gamma p^* - v^*)s + \sigma_1(x,y),0) \bigg] ds \label{eq:l_inteq}
		\end{align}
		where $\sigma_1$ is defined
		\begin{align}
			\sigma_1(x,y) &= \frac{v^* x + (\gamma p^* - v^*) y}{\gamma p^*}
		\end{align}
		Similarly, a direct application of the method of characteristics to \eqref{eq:m_pde} will yield the following piecewise defined integral relation for $m$:
		\begin{align}
			m(x,y) &= \begin{cases} m_1(x,y) & y \leq -\mu^* x + L \\ m_2(x,y) & y > -\mu^* x + L \end{cases} \label{eq:m_inteq}\\
			m_1(x,y) &= \int_0^{\frac{x}{\gamma p^* - v^*}} v^* a(-v^* s + \sigma_2(x,y)) \nonumber\\&\qquad\qquad\qquad\qquad\qquad\times l((\gamma p^* - v^*)s, 0) ds \nonumber \\
			m_2(x,y) &= \int_0^{\frac{L-y}{v^*}} v^* a(-v^* s + L)\nonumber\\&\qquad\qquad\qquad\times l((\gamma p^* - v^*)s + \sigma_3(x,y),0) ds \nonumber
		\end{align}
		where $\sigma_2, \sigma_3$ are
		\begin{align}
			\sigma_2(x,y) &= \mu^* x + y \\
			\sigma_3(x,y) &= x - \frac{1}{\mu^*}(L-y)
		\end{align}

		By substituting in \eqref{eq:m_inteq} into \eqref{eq:l_inteq} and evaluating the result at $y=0$, one generates the following integral equation defined piecewise:
		\begin{align}
		    l(x,0) &= \begin{cases} l_d(x) & x \leq \frac{L}{\mu^*} \\ l_u(x) & x > \frac{L}{\mu^*}\end{cases} \label{eq:lx0}\\
		    l_d(x) &= -\frac{1}{\gamma p^*} c\left( \frac{v^*}{\gamma p^*} x \right) \nonumber\\
		        &\quad+ \int_0^{\frac{\mu^* x}{\gamma p^*}} v^* a(-v^*s + \mu^* x) l((\gamma p^* - v^*)s,0) ds \nonumber\\
		        &\quad+ \int_0^{\frac{x}{\gamma p^*}} \bigg[ k_1 \mu^* c\left(-v^*s + \frac{v^*}{\gamma p^*} x\right) l(\gamma p^* s, 0) \nonumber\\
		        &\qquad\qquad\qquad+ v^* a\left(-v^* s + \frac{v^*}{\gamma p^*} x \right) \nonumber\\
		        &\qquad\qquad\qquad\quad \times l \left( (\gamma p^* - v^*) s + \frac{v^*}{\gamma p^*} x, 0 \right) \bigg] ds \\
		    l_u(x) &= -\frac{1}{\gamma p^*} c\left( \frac{v^*}{\gamma p^*} x \right) \nonumber\\
		        &\quad\nonumber\\[5pt]
		        &\quad+ \int_0^{\frac{L}{v^*} - \frac{1}{\gamma p^*}x} v^* a(-v^*s + L) \nonumber\\
		        &\qquad\qquad\qquad\times l\left((\gamma p^* - v^*)s + x - \frac{L}{\mu^*},0\right) ds \nonumber\\
		        &\quad+ \int_0^{\frac{x}{\gamma p^*}} \bigg[ k_1 \mu^* c\left(-v^*s + \frac{v^*}{\gamma p^*} x\right) l(\gamma p^* s, 0) \nonumber\\
		        &\qquad\qquad\qquad+ v^* a\left(-v^* s + \frac{v^*}{\gamma p^*} x \right) \nonumber\\
		        &\qquad\qquad\qquad\quad \times l \left( (\gamma p^* - v^*) s + \frac{v^*}{\gamma p^*} x, 0 \right) \bigg] ds
		\end{align}
		
		Note that the separate cases of $\mu^* > 1, \mu^* \leq 1$ are self contained in the definition \eqref{eq:lx0}: for $\mu^* \leq 1$, only the condition corresponding to $l_d(x)$ is activated since $x \in (0,L)$, while for $\mu^* > 1$, both cases must be considered.
		
		We will give a rough sketch of the remainder, as it is quite straightforward from \cite{krsticbacksteppingpde,chen2017transportwave,su2017boundary}. By using the method of successive approximations, we can establish an iteration $\{l_n\}_{n=0}^{\infty} \rightarrow l(x,0)$, which can be shown to converge uniformly since the integral equations admitted are affine. This establishes the existence of a solution which can be shown to be unique due to the linearity. The regularity is recovered by noting that the pieces of the solution $l(x,0)$ are $C([0,L])$ compatible. Then the $C(\mathcal{T}_l)$ solutions $k,l$ and $C(\mathcal{T}_u)$ solution $m$ can be directly generated by evaluating $\eqref{eq:k_soln}, \eqref{eq:l_inteq}, \eqref{eq:m_inteq}$, respectively.
	\end{proof}

	\begin{lemma}
		Consider the integral equation given by \eqref{eq:n_ie}, and the definition of $\check{a}$ in \eqref{eq:a_chk_def}. If $\check{a} \in L^2(0,L)$, then there exists a unique $L^2(0,L)$ solution $n$.
	\end{lemma}

	\begin{proof}
		The proof is relatively straightforward, and one may employ standard linear integral equation techniques to recover $n$. Since $\check{a}$ is assumed piecewise continuous, then one such potential method is utilzing the Laplace transform on \eqref{eq:n_ie}:
		\begin{align}
			\hat{n}(s) &= (\gamma p^* - v^*)^{-1} \left[\hat{\check{a}}(s) + \hat{n}(s) \hat{\check{a}}(s) \right]
		\end{align}
		The solution $n \in L^2$ is found by applying the inverse Laplace transform:
		\begin{align}
			n(x) &= \mathcal{L}^{-1} \bigg\{ (\gamma p^* - v^*)^{-1} \left[\hat{\check{a}}(s) + \hat{n}(s) \hat{\check{a}}(s) \right]\bigg\}
		\end{align}
	\end{proof}

	\section{Conclusion}
	In this paper, we have presented a control design for damping app-routing instabilities in a linear ARZ traffic model. The control design is an novel extension to classical infinite-dimensional backstepping, allowing for additional forwarding to account for the non-local boundary condition arising from the routing.

	In this control design, we present stability analysis in $H^1$, which is different than the previous work in $L^2$ stabilizing backstepping control for linear ARZ models. This point is emphasized in the paper, as stability in $L^2$ is, in general, \emph{insufficient} to guarantee existence of solutions in closed loop. In particular, for the ARZ model to be valid, certain pointwise boundedness conditions must be fulfilled, which $L^2$ stability does not necessarily guarantee. Thus, by performing stability analysis in $H^1$, the designer can guarantee existence of solutions on the infinite-time interval by restricting the set of initial conditions for which the control design is valid for. The $H^1$ estimate is conservative, as it is a sufficient condition but not necessary (a less conservative result, for example, may be to consider $W^{1,\infty}$).

	One key weakness to be addressed in this design is the assumption that the app routing $U_{rout}$ is known a priori. In general, the routing is unknown as it involves proprietary algorithms in the separate apps, and may not even be consistent between competing apps. Thus, a natural extension to consider is estimating the app routing disturbance in several various ways. One such case is to presume that the routing feedback is of a linear form (as we have assumed in our paper), and identify the feedback gain $a(y)$ online and combine the identification algorithm with our proposed boundary controller. However, the $H^1$ stability analysis may become extremely convoluted, as we must guarantee $H^1$ bounds such that the model remains valid on the infinite time interval. Another approach may be to treat the routing app as an adversarial controller, assume a bound on app routing disturbance, and then attempt to design a boundary feedback controller subject to bounded state constraints.
	
	\bibliographystyle{ieeetr}
    \bibliography{ref}

\begin{thebibliography}{10}

\bibitem{lighthill1955kinematic}
M.~J. Lighthill and G.~B. Whitham, ``On kinematic waves ii. a theory of traffic
  flow on long crowded roads,'' {\em Proceedings of the Royal Society of
  London. Series A. Mathematical and Physical Sciences}, vol.~229, no.~1178,
  pp.~317--345, 1955.

\bibitem{richards1956shock}
P.~I. Richards, ``Shock waves on the highway,'' {\em Operations research},
  vol.~4, no.~1, pp.~42--51, 1956.

\bibitem{aw2000resurrection}
A.~Aw and M.~Rascle, ``Resurrection of" second order" models of traffic flow,''
  {\em SIAM journal on applied mathematics}, vol.~60, no.~3, pp.~916--938,
  2000.

\bibitem{zhang2002non}
H.~M. Zhang, ``A non-equilibrium traffic model devoid of gas-like behavior,''
  {\em Transportation Research Part B: Methodological}, vol.~36, no.~3,
  pp.~275--290, 2002.

\bibitem{yu2019traffic}
H.~Yu and M.~Krstic, ``Traffic congestion control for aw--rascle--zhang
  model,'' {\em Automatica}, vol.~100, pp.~38--51, 2019.

\bibitem{8550563}
I.~{Karafyllis}, N.~{Bekiaris-Liberis}, and M.~{Papageorgiou}, ``Traffic flow
  inspired analysis and boundary control for a class of 2x2 hyperbolic
  systems,'' in {\em 2018 European Control Conference (ECC)}, pp.~1314--1321,
  June 2018.

\bibitem{BEKIARISLIBERIS20191}
N.~Bekiaris-Liberis and A.~Delis, ``Feedback control of freeway traffic flow
  via time-gap manipulation of acc-equipped vehicles: A pde-based approach,''
  {\em IFAC-PapersOnLine}, vol.~52, no.~6, pp.~1 -- 6, 2019.
\newblock 1st IFAC Workshop on Control of Transportation Systems WCTS 2019.

\bibitem{krsticbacksteppingpde}
M.~Krstic and A.~Smyshlyaev, {\em Boundary Control of {PDE}s: A Course on
  Backstepping Designs}.
\newblock SIAM, 2008.

\bibitem{dimeglio2013nplus1}
F.~D. Meglio, R.~Vazquez, and M.~Krstic, ``Stabilization of a system of coupled
  first-order hyperbolic linear pdes with a single boundary input,'' {\em IEEE
  Transactions on Automatic Control}, vol.~58, no.~12, pp.~3097--3111, 2013.

\bibitem{hu2015nm}
L.~Hu, F.~D. Meglio, R.~Vazquez, and M.~Krstic, ``Control of homodirectional
  and general heterodirectional linear coupled hyperbolic {PDEs},'' {\em IEEE
  Transactions on Automatic Control}, vol.~61, no.~10, pp.~3301--3314, 2016.

\bibitem{coron2013local}
J.-M. Coron, R.~Vazquez, M.~Krstic, and G.~Bastin, ``Local exponential h2
  stabilization of a 2x2 quasilinear hyperbolic system using backstepping,''
  {\em SIAM Journal on Control and Optimization}, vol.~51, no.~3,
  pp.~2005--2035, 2013.

\bibitem{chen2017transportwave}
S.~Chen, R.~Vazquez, and M.~Krstic, ``Stabilization of an underactuated coupled
  transport-wave {PDE} system,'' in {\em 2017 American Control Conference
  (ACC)}, pp.~2504--2509, May 2017.

\bibitem{su2017boundary}
L.~Su, W.~Guo, J.-M. Wang, and M.~Krstic, ``Boundary stabilization of wave
  equation with velocity recirculation,'' {\em IEEE Transactions on Automatic
  Control}, vol.~62, no.~9, pp.~4760--4767, 2017.

\end{thebibliography}

	%
	%
\end{document}